\documentclass[11pt]{article}

\usepackage[a4paper,margin=1.08in]{geometry}
\usepackage{amsmath,amssymb,amsthm,mathtools}
\usepackage{microtype}
\usepackage{hyperref}
\hypersetup{hidelinks}

\newtheorem{theorem}{Theorem}[section]
\newtheorem{proposition}[theorem]{Proposition}
\newtheorem{lemma}[theorem]{Lemma}
\newtheorem{corollary}[theorem]{Corollary}

\theoremstyle{remark}
\newtheorem{remark}[theorem]{Remark}

\newcommand{\PP}{\mathbb P}
\newcommand{\EE}{\mathbb E}
\newcommand{\Poi}{\operatorname{Poisson}}
\newcommand{\PPP}{\operatorname{PPP}}
\newcommand{\diam}{\operatorname{diam}}
\newcommand{\dist}{\operatorname{dist}}
\newcommand{\ord}{\operatorname{ord}}
\newcommand{\cQ}{\mathcal Q}
\newcommand{\cE}{\mathcal E}
\newcommand{\cL}{\mathcal L}
\newcommand{\cR}{\mathcal R}
\newcommand{\cX}{\mathcal X}
\newcommand{\Gum}{G_{\mathrm{Gum}}}
\newcommand{\doi}[1]{\href{https://doi.org/#1}{doi:#1}}

\title{Exact Diameter Windows for Random Cayley Graphs\\
on Odd-Order Abelian Groups}

\author{Yao Zhi}
\date{}

\begin{document}

\maketitle

\begin{abstract}
	Let \(d\ge2\) be fixed and let \(G_n\) be finite abelian groups of odd
	orders \(N_n\to\infty\).  We determine the centered diameter-\(d\)
	critical window for the standard random Cayley graph in which each
	nonzero group element is selected independently.  Writing
	\(M_n=(N_n-1)/2\), we prove that the normalized first
	distance-\(d\) coverage times satisfy
	\[
	\sum_{[x]\in(G_n\setminus\{0\})/\{\pm1\}}
	\delta_{\frac{N_n^{d-1}}{d!}\tau_{n,[x]}^d-\log M_n}
	\xrightarrow{d}
	\PPP(e^{-z}\,dz).
	\]
	Consequently, the number of antipodal defects in the critical window
	converges in total variation to a Poisson law, the diameter transition
	has the Gumbel profile \(e^{-e^{-c}}\), and the diameter hitting time
	has Gumbel fluctuations.  In the original generator-density
	parametrization this yields the sharp fixed-\(d\) threshold constant
	\(d!/2^d\) throughout the odd-order abelian class.  For \(d=2\), we
	additionally obtain an exact path--cycle decomposition of the target
	representation graphs.
\end{abstract}

\medskip
\noindent\textbf{Keywords.} Random Cayley graph; diameter; critical window;
Poisson point process; Gumbel limit; finite abelian group.

\smallskip
\noindent\textbf{2020 Mathematics Subject Classification.} Primary 05C80;
Secondary 05C12, 05C25, 60G55.

\section{Introduction and main results}
\label{sec:intro}

Random Cayley graphs provide a natural setting in which algebraic symmetry
interacts with probabilistic dependence.  In the independent-density model
\(\mathcal G(G,p)\), each group element is placed in a random set independently
with probability \(p\), and this set together with its inverses generates an
undirected Cayley graph; the identity is immaterial.  For diameter two,
Christofides and Markstr\"om established general threshold bounds and showed
that the threshold constant can depend on the underlying family of groups
\cite{ChristofidesMarkstrom2014,ChristofidesMarkstromRange2014}.  Related
metric questions in models with a prescribed or random number of generators
have been studied for cyclic, abelian, nilpotent and nonabelian families; see,
for example,
\cite{MarklofStrombergsson2013,HermonOleskerTaylor2023,ElBazPagano2021,Sardari2019}.
These models are distinct from the independent-density model considered here.

Recently, Christofides, Markstr\"om and Savvidou developed a
representation-hypergraph method for diameter \(d\) and proved threshold-scale
results uniformly for
\[
    2\le d\le
    (1-\gamma)\sqrt{\frac{\log N}{2\log\log N}},
\]
where \(\gamma\in(0,1)\) is fixed
\cite{ChristofidesMarkstromSavvidou2026}.  For abelian groups their lower
threshold has constant \(d!/2^d\), while the matching upper theorem with this
constant is stated for cyclic groups
\cite[Theorems~1.3 and~1.5]{ChristofidesMarkstromSavvidou2026}.  These results
locate the transition on the multiplicative threshold scale; they do not
identify the centered critical-window law or the extremal process of the last
uncovered targets.

Our purpose is to determine this finer limit for every fixed \(d\ge2\) over
the entire class of finite abelian groups of odd order.  Theorem~\ref{thm:main-ppp}
shows that the normalized first distance-\(d\) coverage times converge to a
Poisson point process with intensity \(e^{-z}\,dz\).  Consequently the number
of antipodal defects has a Poisson limit in total variation, the diameter
transition has profile \(e^{-e^{-c}}\), and the diameter hitting time has
Gumbel fluctuations.  In the original generator-density parametrization,
Corollary~\ref{cor:original-q} gives, for every fixed \(d\),
\[
    q^d\sim \frac{d!}{2^d}\frac{\log N}{N^{d-1}}
\]
throughout the odd-order abelian class.  Thus the matching constant from the
cyclic threshold theorem extends to this class, and the threshold statement is
refined to a centered extremal-process limit.

The top-layer representation hypergraphs and their Janson dependency estimates
are closely related to the framework of
\cite{ChristofidesMarkstromSavvidou2026}.  Two additional issues become
essential at the centered scale.  First, the leading representation count must
be accurate enough that its relative error remains negligible after
multiplication by \(\log N\).  Second, representations using fewer than \(d\)
inverse classes may be highly nonuniform because of torsion.  We resolve the
first issue by proving the uniform expansion
\(N^{d-1}/d!+O_d(N^{d-2})\) for top supports together with an
\(O_d(N^{d-2})\) overlap bound.  For the second, rather than classifying the
relevant torsion strata, we prove that for each \(s<d\) the total number of
\(s\)-supports over all targets is \(O_d(N^s)\).  This average estimate removes
only \(o(N)\) exceptional targets; their probability of appearing in any fixed
critical window is itself \(o(1)\).  The resulting uniform multitime
factorization is what yields the Poisson point process, rather than only a
one-time threshold estimate.

Odd order enters in two essential places.  Every nonzero element then belongs
to a two-element inversion class, and multiplication by \(2\) is an
automorphism.  The latter makes the two-variable systems arising from distinct
sign patterns uniformly invertible and is the algebraic reason the overlap
bounds are independent of the structure of the group.

The case \(d=2\) has additional exact structure.  If \(m=\ord(x)\),
Theorem~\ref{thm:d2-geometry} shows that the representation graph of \(x\) on
the inversion quotient is
\[
    P_{(m-1)/2}\sqcup\frac{N/m-1}{2}\,C_m.
\]
Although this graph depends on the target order, it always has exactly
\((N-3)/2\) edges, and every two distinct target graphs share exactly one edge.
Thus the general fixed-\(d\) argument gives the universal asymptotic process,
while diameter two admits a strictly sharper microscopic description.

The remainder of the paper is organized as follows.  Section~\ref{sec:technical}
develops the representation-support estimates, the regular--exceptional
decomposition, and the uniform multitime factorization in the critical
window.  Section~\ref{sec:extremal} derives the Poisson extremal process and
the diameter-window and hitting-time consequences.  Section~\ref{sec:d2}
gives the exact diameter-two representation geometry.

\subsection{Model and main statements}

Let \(G\) be a finite abelian group, written additively, of odd order
\(N\).  Let \(S(q)\subseteq G\setminus\{0\}\) contain each nonzero
element independently with probability \(q\), and define
\[
    \Gamma_G(q)
    :=
    \operatorname{Cay}\bigl(G,S(q)\cup(-S(q))\bigr).
\]
Put
\[
    \cQ_G:=(G\setminus\{0\})/\{\pm1\},
    \qquad
    [a]:=\{a,-a\},
    \qquad
    M:=|\cQ_G|=\frac{N-1}{2}.
\]
Since \(N\) is odd, every class has size two.  An inverse class is
active once at least one of its two elements is selected, so its
activation probability is
\begin{equation}
    r(q)=1-(1-q)^2=2q-q^2.
    \label{eq:r-q}
\end{equation}

We first work in pair-density time.  Let
\((T_\xi)_{\xi\in\cQ_G}\) be independent
\(\operatorname{Unif}[0,1]\) variables, declare \(\xi\) active at
time \(r\) when \(T_\xi\le r\), and write \(\Gamma_G^*(r)\) for the
resulting Cayley graph.  For each fixed \(q\),
\(\Gamma_G(q)\) and \(\Gamma_G^*(r(q))\) have the same law.

Fix \(d\ge2\).  For \([x]\in\cQ_G\), define
\[
    \tau_{G,[x]}
    :=
    \inf\{r:\dist_{\Gamma_G^*(r)}(0,x)\le d\}.
\]
For a sequence \(G_n\) of odd orders \(N_n\to\infty\), put
\(M_n=(N_n-1)/2\) and
\begin{equation}
    \Xi_n
    :=
    \sum_{[x]\in\cQ_{G_n}}\delta_{Z_{n,[x]}},
    \qquad
    Z_{n,[x]}
    :=
    \frac{N_n^{d-1}}{d!}\tau_{G_n,[x]}^d-\log M_n.
    \label{eq:defect-process}
\end{equation}

\begin{theorem}[Universal fixed-diameter extremal process]
\label{thm:main-ppp}
Let \(d\ge2\) be fixed.  For every sequence of finite abelian groups
\(G_n\) of odd orders \(N_n\to\infty\),
\[
    \Xi_n\xrightarrow{d}\PPP(e^{-z}\,dz)
\]
in the space of locally finite point measures on \(\mathbb R\) with
the vague topology.
\end{theorem}

For a group of order \(N\), define
\begin{equation}
    r_N(c)
    :=
    \left(
        \frac{d!(\log M+c)}{N^{d-1}}
    \right)^{1/d},
    \qquad M=\frac{N-1}{2},
    \label{eq:critical-r}
\end{equation}
whenever the expression is positive.  For the sequence \(G_n\), write
\(r_n(c):=r_{N_n}(c)\), and let
\[
    V_n(c)
    :=
    \#\bigl\{
        [x]\in\cQ_{G_n}:
        \dist_{\Gamma_{G_n}^*(r_n(c))}(0,x)>d
    \bigr\}.
\]

\begin{corollary}[Exact diameter-\(d\) window]
\label{cor:exact-window}
For every fixed \(c\in\mathbb R\),
\[
    d_{\mathrm{TV}}
    \bigl(\mathcal L(V_n(c)),\Poi(e^{-c})\bigr)
    \longrightarrow0.
\]
Moreover,
\[
    \PP\bigl(\diam\Gamma_{G_n}^*(r_n(c))\le d\bigr)
    \longrightarrow
    \exp\{-e^{-c}\}.
\]
If \(W_n(c)\) is the number of nonzero vertices at distance greater
than \(d\) from \(0\), then \(W_n(c)=2V_n(c)\) identically.
\end{corollary}

Let
\[
    R_n^*
    :=
    \inf\{r:\diam\Gamma_{G_n}^*(r)\le d\}.
\]

\begin{corollary}[Gumbel hitting time]
\label{cor:gumbel-pair}
If \(\Gum\) has distribution function
\(\PP(\Gum\le c)=\exp\{-e^{-c}\}\), then
\[
    \frac{N_n^{d-1}}{d!}(R_n^*)^d-\log M_n
    \xrightarrow{d}
    \Gum.
\]
\end{corollary}

\begin{corollary}[Original generator-density parametrization]
\label{cor:original-q}
Let
\[
    Q_n^*
    :=
    \inf\{q:\diam\Gamma_{G_n}(q)\le d\}.
\]
Then
\[
    \frac{2^d}{d!}N_n^{d-1}(Q_n^*)^d-\log N_n
    \xrightarrow{d}
    \Gum-\log2.
\]
If
\begin{equation}
    \frac{2^d}{d!}N_n^{d-1}q_n^d
    =
    \log N_n+c+o(1),
    \label{eq:q-window}
\end{equation}
then, writing \(V_n^{(q)}\) for the number of antipodal defect
classes in \(\Gamma_{G_n}(q_n)\),
\[
    d_{\mathrm{TV}}
    \bigl(\mathcal L(V_n^{(q)}),\Poi(e^{-c}/2)\bigr)
    \longrightarrow0
\]
and
\[
    \PP\bigl(\diam\Gamma_{G_n}(q_n)\le d\bigr)
    \longrightarrow
    \exp\{-e^{-c}/2\}.
\]
In particular,
\[
    q^d\sim
    \frac{d!}{2^d}\frac{\log N}{N^{d-1}}
\]
is the sharp fixed-\(d\) threshold throughout the odd-order abelian
class.
\end{corollary}

\section{Representation supports and critical-window estimates}
\label{sec:technical}

Throughout Sections~\ref{sec:technical} and~\ref{sec:extremal}, \(d\ge2\)
is fixed, \(G\) is an abelian group of odd order \(N\), and
\(M=(N-1)/2\).

\begin{lemma}
\label{lem:odd-order}
The doubling map \(a\mapsto2a\) is an automorphism of \(G\), and
inversion has no fixed point in \(G\setminus\{0\}\).
\end{lemma}

\begin{proof}
If \(2a=0\), then the order of \(a\) divides both \(2\) and \(N\),
so \(a=0\).  Thus doubling is injective and hence bijective.
The assertion about inversion is equivalent.
\end{proof}

For \(x\ne0\) and \(1\le s\le d\), let \(\cL_s(x)\) be the family of
\(s\)-element supports
\(E=\{\xi_1,\ldots,\xi_s\}\subseteq\cQ_G\) for which there exist
representatives \(a_i\in\xi_i\) and positive integers \(m_i\) such that
\[
    m_1+\cdots+m_s\le d,
    \qquad
    m_1a_1+\cdots+m_sa_s=x.
\]
Write \(\ell_s(x):=|\cL_s(x)|\), and set
\(\cE_d(x):=\cL_d(x)\) and \(e_d(x):=|\cE_d(x)|\).

\begin{lemma}[Support reduction]
For \(x\ne0\),
\[
    \dist_{\Gamma_G^*(r)}(0,x)\le d
\]
if and only if some support in
\(\bigcup_{s=1}^d\cL_s(x)\) is entirely active.  Moreover,
\(E\in\cE_d(x)\) precisely when \(E\) consists of \(d\) distinct
inverse classes admitting representatives
\(a_1,\ldots,a_d\) with \(a_1+\cdots+a_d=x\).
\end{lemma}

\begin{proof}
Consider a walk of length at most \(d\) from \(0\) to \(x\).
For each inverse class used by the walk, choose a representative
\(a_i\), and let \(u_i\) and \(v_i\) be the numbers of steps in the
directions \(a_i\) and \(-a_i\).  Delete classes with \(u_i=v_i\);
for each remaining class reverse \(a_i\) if necessary and put
\(m_i=|u_i-v_i|\).  Then \(m_i\ge1\),
\[
    x=\sum_i m_i a_i,
    \qquad
    \sum_i m_i\le\sum_i(u_i+v_i)\le d.
\]
Thus the active classes contain a support in some \(\cL_s(x)\).

Conversely, a relation
\(m_1a_1+\cdots+m_sa_s=x\) with \(\sum_i m_i\le d\) gives a walk
of length at most \(d\) by taking \(m_i\) steps in direction \(a_i\).
When \(s=d\), positivity of the \(m_i\)'s forces
\(m_1=\cdots=m_d=1\).
\end{proof}

\subsection{Top supports and overlaps}

\begin{proposition}[Uniform top-support count]
\label{prop:top-supports}
Uniformly over finite abelian groups \(G\) of odd order \(N\) and
\(x\ne0\),
\[
    e_d(x)=\frac{N^{d-1}}{d!}+O_d(N^{d-2}).
\]
\end{proposition}

\begin{proof}
Let \(T_x\) be the ordered \(d\)-tuples
\((a_1,\ldots,a_d)\in G^d\) satisfying
\(a_1+\cdots+a_d=x\), with every \(a_i\ne0\) and
\([a_i]\ne[a_j]\) for \(i\ne j\).
Without these restrictions there are exactly \(N^{d-1}\) solutions.
Tuples with \(a_i=0\) contribute \(O_d(N^{d-2})\).  The same bound
holds for \(a_i=a_j\): after the other \(d-2\) variables are fixed,
one obtains an equation \(2a_i=b\), which has a unique solution by
Lemma~\ref{lem:odd-order}.  If \(a_i=-a_j\), then for \(d=2\) the
target equation would force \(x=0\), while for \(d\ge3\) one may
choose \(a_i\) and \(d-3\) of the remaining entries freely and the
last entry is determined by the target equation.  This again gives
\(O_d(N^{d-2})\) tuples.  Hence
\[
    |T_x|=N^{d-1}+O_d(N^{d-2}).
\]

It remains to control multiple orientations of a quotient support.
Suppose \(a_1+\cdots+a_d=x\) and reversing the signs on a nonempty
proper set \(S\subset[d]\) gives another representation of \(x\).
Then
\[
    2\sum_{i\in S}a_i=0,
\]
and Lemma~\ref{lem:odd-order} gives
\(\sum_{i\in S}a_i=0\).  For fixed \(S\), after \(d-2\) variables
are fixed, one remaining variable in \(S\) and one in \(S^c\) are
uniquely determined.  Thus only \(O_d(N^{d-2})\) ordered tuples are
involved in orientation collisions.  Consequently
\[
    |T_x|=d!\,e_d(x)+O_d(N^{d-2}),
\]
which proves the claim.
\end{proof}

\begin{lemma}[Common top supports]
\label{lem:common-top}
If \([x]\ne[y]\), then
\[
    |\cE_d(x)\cap\cE_d(y)|=O_d(N^{d-2})
\]
uniformly in \(G,x,y\).
\end{lemma}

\begin{proof}
For a common support choose an ordering and orientation with
\(a_1+\cdots+a_d=x\).  Since the same inverse classes also represent
\(y\), there is \(\sigma\in\{\pm1\}^d\) such that
\[
    \sigma_1a_1+\cdots+\sigma_da_d=y.
\]
The sign vector is not constant, since \([x]\ne[y]\).
Choose \(j,k\) with \(\sigma_j=1\) and \(\sigma_k=-1\).
After the other \(d-2\) variables are fixed, \(a_j,a_k\) satisfy a
system with coefficient matrix
\[
    \begin{pmatrix}1&1\\1&-1\end{pmatrix}.
\]
Its determinant is \(-2\), so Lemma~\ref{lem:odd-order} makes the
system uniquely solvable in \(G^2\).  There are only \(O_d(1)\)
choices of signs and orderings.
\end{proof}

The following estimate is the inversion-quotient analogue, for fixed \(d\),
of the intersecting-path count underlying
\cite[Lemma~3.1]{ChristofidesMarkstromSavvidou2026}.

\begin{lemma}[Intersecting top supports]
\label{lem:dependency-count}
Fix \(k\ge1\) and pairwise distinct target classes
\([x_1],\ldots,[x_k]\).  For \(1\le j\le d-1\), the number of ordered
pairs \((E,F)\) of distinct supports from
\(\bigcup_{i=1}^k\cE_d(x_i)\) with \(|E\cap F|=j\) is
\[
    O_{d,k}(N^{2d-j-2}).
\]
\end{lemma}

\begin{proof}
It is enough to count signed ordered representation pairs.  First
choose the target of the first support and a signed ordered
representation \(a_1+\cdots+a_d=x_i\); this gives
\(O_{d,k}(N^{d-1})\) possibilities.  Choose the \(j\) shared
positions in the two representations, a bijection between them, and
the relative signs; these choices contribute only a factor depending
on \(d\).  The second representation then has \(d-j\) new
entries.  After \(d-j-1\) of them are chosen freely in \(G\), the
target equation determines the last one.  Ignoring the requirements
that the resulting inverse classes be nonzero, new and distinct can
only increase the count.  Thus the number of ordered support pairs is
\(O_{d,k}(N^{d-1}N^{d-j-1})\), as claimed.
\end{proof}

\subsection{Lower supports and exceptional targets}

\begin{lemma}[Average lower-support bound]
\label{lem:lower-average}
For every \(1\le s<d\),
\[
    \sum_{[x]\in\cQ_G}\ell_s(x)=O_d(N^s).
\]
\end{lemma}

\begin{proof}
Fix an \(s\)-element support
\(E=\{\xi_1,\ldots,\xi_s\}\).  There are \(O_d(1)\) positive integer
vectors \((m_1,\ldots,m_s)\) with \(\sum_i m_i\le d\), and
\(2^s=O_d(1)\) choices of representatives \(a_i\in\xi_i\).  Each
choice determines at most one target class through
\(x=m_1a_1+\cdots+m_sa_s\).  Thus a fixed support belongs to
\(O_d(1)\) target families.  Since there are
\(\binom Ms=O_d(N^s)\) supports, the result follows.
\end{proof}

Fix a compact set \(K\subset\mathbb R\) and put
\[
    r_*:=\max_{c\in K}r_N(c),
    \qquad
    \beta_N:=\max_{1\le s<d}N^{s-1}r_*^s.
\]
Since \(d\) is fixed,
\[
    \beta_N
    =
    O_d\left(
        \max_{1\le s<d}
        (\log N)^{s/d}N^{-(d-s)/d}
    \right)
    =o(1).
\]
Let \(H_N:=\beta_N^{-1/2}\).  Call \([x]\) regular if
\[
    \ell_s(x)\le H_NN^{s-1}
    \qquad(1\le s<d),
\]
and exceptional otherwise.  Write \(\cR_N\) and \(\cX_N\) for the
two sets; their dependence on the fixed compact set \(K\) is suppressed.

\begin{lemma}[Lower-support peeling]
\label{lem:peeling}
We have \(|\cX_N|=o(N)\), and uniformly over regular targets,
\[
    \sum_{s<d}\ell_s(x)r_*^s=o(1).
\]
\end{lemma}

\begin{proof}
Lemma~\ref{lem:lower-average} and Markov's inequality give
\[
    |\cX_N|
    \le
    \sum_{s<d}\frac{O_d(N^s)}{H_NN^{s-1}}
    =
    O_d(N/H_N)
    =
    o(N).
\]
For a regular target,
\[
    \sum_{s<d}\ell_s(x)r_*^s
    \le
    H_N\sum_{s<d}N^{s-1}r_*^s
    \le
    (d-1)H_N\beta_N
    =
    O_d(\sqrt{\beta_N})
    =
    o(1).
\]
\end{proof}

\subsection{Uniform multitime factorization}

For \(x\ne0\), write
\[
    D_x(r):=\{\dist_{\Gamma_G^*(r)}(0,x)>d\}.
\]

\begin{proposition}[Uniform multitime factorization]
\label{prop:multitime}
Fix \(k\ge1\) and a compact set \(K\subset\mathbb R\).
Let \([x_1],\ldots,[x_k]\in\cR_N\) be pairwise distinct and let
\(c_1,\ldots,c_k\in K\).  Put \(r_i=r_N(c_i)\).  Then, uniformly in
the group, the regular targets and the \(c_i\)'s,
\[
    \PP\left(\bigcap_{i=1}^kD_{x_i}(r_i)\right)
    =
    M^{-k}
    \exp\left\{-\sum_{i=1}^kc_i\right\}(1+o(1)).
\]
\end{proposition}

\begin{proof}
For each \(i\) and \(E\in\cE_d(x_i)\), let \(A(E,i)\) be the event
that every coordinate of \(E\) is active by time \(r_i\).  If the
same support occurs for several targets, the corresponding events are
nested; retain only the one with the largest threshold.  Let
\(\mathcal A\) be the resulting family and put
\(\mu=\sum_{A\in\mathcal A}\PP(A)\).

By Proposition~\ref{prop:top-supports} and
\eqref{eq:critical-r},
\[
    e_d(x_i)r_i^d=\log M+c_i+o(1).
\]
Lemma~\ref{lem:common-top} shows that merging supports common to
different targets changes the sum of the individual means by at most
\[
    O_{d,k}(N^{d-2}r_*^d)
    =
    O_{d,k}\left(\frac{\log N}{N}\right)
    =
    o(1).
\]
Hence
\[
    \mu=k\log M+\sum_{i=1}^kc_i+o(1).
\]

Each event in \(\mathcal A\) is decreasing in the usual
coordinatewise order on the activation times.  Harris's correlation
inequality applies to product measures, and Riordan and Warnke's
Janson inequality explicitly applies equally to decreasing events
\cite{Harris1960,RiordanWarnke2015}.  With \(\Delta\) denoting the
sum of \(\PP(A\cap B)\) over ordered pairs of distinct dependent
events,
\[
    \prod_{A\in\mathcal A}(1-\PP(A))
    \le
    \PP\left(\bigcap_{A\in\mathcal A}A^c\right)
    \le
    \exp\{-\mu+\Delta/2\}.
\]
Moreover,
\[
    \sum_{A\in\mathcal A}\PP(A)^2
    =
    O_{d,k}(N^{d-1}r_*^{2d})
    =
    O_{d,k}\left(\frac{(\log N)^2}{N^{d-1}}\right)
    =
    o(1),
\]
so the product on the left equals \(\exp\{-\mu+o(1)\}\).

If two distinct top supports meet in \(j\) inverse classes, their
joint activation probability is at most \(r_*^{2d-j}\).
Lemma~\ref{lem:dependency-count} therefore gives
\[
    \Delta
    \le
    C_{d,k}\sum_{j=1}^{d-1}
    N^{2d-j-2}r_*^{2d-j}.
\]
For \(1\le j<d\),
\[
    N^{2d-j-2}r_*^{2d-j}
    =
    O_d\left(
        N^{-j/d}(\log N)^{(2d-j)/d}
    \right)
    =
    o(1),
\]
and hence \(\Delta=o(1)\).  Therefore the probability
\(P_{\mathrm{top}}\) that no top bad event occurs satisfies
\[
    P_{\mathrm{top}}
    =
    M^{-k}\exp\left\{-\sum_{i=1}^kc_i\right\}(1+o(1)).
\]

Let \(B_{\mathrm{low}}\) be the event that none of the lower supports
of the \(k\) targets is fully active by its corresponding threshold.
Lemma~\ref{lem:peeling} and the union bound give
\[
    \PP(B_{\mathrm{low}}^c)
    \le
    \sum_{i=1}^k\sum_{s<d}\ell_s(x_i)r_*^s
    =
    o(1).
\]
Both top avoidance and \(B_{\mathrm{low}}\) are increasing events in
the activation-time coordinates, so Harris's inequality yields
\[
    P_{\mathrm{top}}\PP(B_{\mathrm{low}})
    \le
    \PP\left(\bigcap_{i=1}^kD_{x_i}(r_i)\right)
    \le
    P_{\mathrm{top}}.
\]
Since \(\PP(B_{\mathrm{low}})=1-o(1)\), the proposition follows.
\end{proof}

\begin{lemma}[Exceptional targets vanish]
\label{lem:exceptional-vanish}
Fix a compact set \(K\subset\mathbb R\).  Uniformly for \(c\in K\),
\[
    \EE\#\{[x]\in\cX_N:D_x(r_N(c))\}=o(1).
\]
Consequently, with probability \(1-o(1)\), no exceptional target
produces a point of the defect process in any prescribed bounded
interval.
\end{lemma}

\begin{proof}
For every target \(x\), \(D_x(r)\) implies that no top support in
\(\cE_d(x)\) is active.  The single-target top-layer Janson estimate,
obtained from the preceding proof with \(k=1\), and
Proposition~\ref{prop:top-supports} give uniformly for \(c\in K\)
\[
    \PP(D_x(r_N(c)))\le\frac{C_K}{M}.
\]
Since \(|\cX_N|=o(M)\) by Lemma~\ref{lem:peeling}, the expectation is
\(o(1)\).

If \(I=(a,b]\) is bounded and an exceptional target contributes a
point \(Z_{[x]}\in I\), then \(Z_{[x]}>a\), equivalently
\(D_x(r_N(a))\) occurs.  Markov's inequality gives the second
assertion; finite unions of bounded intervals are identical.
\end{proof}

\section{The extremal process}
\label{sec:extremal}

\begin{proof}[Proof of Theorem~\ref{thm:main-ppp}]
Write \(M=M_n\).  By \eqref{eq:defect-process} and
\eqref{eq:critical-r},
\[
    \{Z_{n,[x]}>c\}=D_x(r_n(c)).
\]

Let \(B\) be a finite disjoint union of bounded half-open intervals
and choose a compact set \(K\) containing all their endpoints.
Let \(\cR_N\) and \(\cX_N\) be the corresponding regular and
exceptional target sets.

For bounded intervals \(I_i=(a_i,b_i]\), inclusion--exclusion and
Proposition~\ref{prop:multitime} give, uniformly over pairwise
distinct regular targets,
\[
    \PP\bigl(
        Z_{n,[x_i]}\in I_i,\ 1\le i\le k
    \bigr)
    =
    M^{-k}\prod_{i=1}^k(e^{-a_i}-e^{-b_i})+o(M^{-k}).
\]
Put
\[
    \nu(B):=\int_Be^{-z}\,dz
\]
and
\[
    N_n^{\mathrm{reg}}(B)
    :=
    \#\{[x]\in\cR_N:Z_{n,[x]}\in B\}.
\]
Since \(|\cR_N|=M-o(M)\), summing over distinct regular targets gives
for every fixed \(k\ge1\)
\[
    \EE\bigl(N_n^{\mathrm{reg}}(B)\bigr)_k
    \longrightarrow
    \nu(B)^k.
\]

The Bonferroni inequalities applied to these factorial moments imply
\[
    \PP(N_n^{\mathrm{reg}}(B)=0)\longrightarrow e^{-\nu(B)}.
\]
Indeed, truncate the alternating factorial-moment expansion above and
below, let \(n\to\infty\), and then let the truncation level tend to
infinity.

Let
\[
    N_n^{\mathrm{exc}}(B)
    :=
    \#\{[x]\in\cX_N:Z_{n,[x]}\in B\}.
\]
Lemma~\ref{lem:exceptional-vanish} gives
\[
    \EE N_n^{\mathrm{exc}}(B)=o(1),
    \qquad
    \PP(N_n^{\mathrm{exc}}(B)>0)=o(1).
\]
Thus, for \(N_n(B):=\Xi_n(B)\),
\[
    \EE N_n(B)\longrightarrow\nu(B),
    \qquad
    \PP(N_n(B)=0)\longrightarrow e^{-\nu(B)}.
\]
The Poisson point-process convergence criterion
\cite{Kallenberg2021}, applied to finite unions of bounded half-open
intervals, yields
\[
    \Xi_n\xrightarrow{d}\PPP(e^{-z}\,dz).
\]
\end{proof}

\begin{proof}[Proof of Corollary~\ref{cor:exact-window}]
Fix \(c\) and choose a compact set \(K\) containing \(c\).  Write
\[
    V_n(c)=V_n^{\mathrm{reg}}(c)+V_n^{\mathrm{exc}}(c).
\]
For every fixed \(k\ge1\), Proposition~\ref{prop:multitime} gives
\[
    \EE\bigl(V_n^{\mathrm{reg}}(c)\bigr)_k
    =
    (|\cR_N|)_kM^{-k}e^{-kc}(1+o(1))
    \longrightarrow e^{-kc}.
\]
The factorial-moment criterion
\cite{JansonLuczakRucinski2000} gives
\(V_n^{\mathrm{reg}}(c)\xrightarrow{d}\Poi(e^{-c})\).
Since the variables are integer valued, their probability mass
functions converge pointwise; Scheff\'e's lemma then upgrades this to
total-variation convergence.

By Lemma~\ref{lem:exceptional-vanish},
\[
    \PP(V_n^{\mathrm{exc}}(c)>0)
    \le
    \EE V_n^{\mathrm{exc}}(c)
    =
    o(1).
\]
Hence
\[
    d_{\mathrm{TV}}
    \bigl(\mathcal L(V_n(c)),\Poi(e^{-c})\bigr)
    \longrightarrow0.
\]
Finally,
\[
    \PP\bigl(\diam\Gamma_{G_n}^*(r_n(c))\le d\bigr)
    =
    \PP(V_n(c)=0)
    \longrightarrow
    e^{-e^{-c}}.
\]
Inversion is a graph automorphism fixing \(0\), so \(x\) is a defect
if and only if \(-x\) is, and \(W_n(c)=2V_n(c)\).
\end{proof}

\begin{proof}[Proof of Corollary~\ref{cor:gumbel-pair}]
For fixed \(c\),
\[
    \left\{
        \frac{N_n^{d-1}}{d!}(R_n^*)^d-\log M_n\le c
    \right\}
    =
    \{V_n(c)=0\}.
\]
Corollary~\ref{cor:exact-window} gives convergence to
\(\exp\{-e^{-c}\}\), the standard Gumbel distribution function.
\end{proof}

\begin{proof}[Proof of Corollary~\ref{cor:original-q}]
Let \(U_a\), \(a\in G_n\setminus\{0\}\), be independent
\(\operatorname{Unif}[0,1]\) variables and expose
\(S_n(q)=\{a:U_a\le q\}\).  For an inverse class \(\xi=[a]\), put
\(\widehat Q_\xi=\min\{U_a,U_{-a}\}\) and
\(T_\xi=r(\widehat Q_\xi)\).  Since
\(\PP(\widehat Q_\xi\le q)=r(q)\), the variables \(T_\xi\) are
independent and uniform on \([0,1]\).  Hence the element-density and
pair-density processes may be coupled pathwise so that
\[
    R_n^*=r(Q_n^*)=2Q_n^*-(Q_n^*)^2.
\]
Since \(r(q)=2q-q^2\ge q\) on \([0,1]\), the coupling gives
\(Q_n^*\le R_n^*\).  Corollary~\ref{cor:gumbel-pair} therefore yields
\[
    Q_n^*
    =
    O_{\PP}\left(
        N_n^{-(d-1)/d}(\log N_n)^{1/d}
    \right).
\]
Moreover \(N_n^{d-1}(Q_n^*)^d=O_{\PP}(\log N_n)\), whence
\[
    N_n^{d-1}(Q_n^*)^{d+1}
    =
    O_{\PP}\left(
        N_n^{-(d-1)/d}(\log N_n)^{1+1/d}
    \right)
    =
    o_{\PP}(1).
\]
Therefore
\[
    \frac{N_n^{d-1}}{d!}
    \bigl(2Q_n^*-(Q_n^*)^2\bigr)^d
    =
    \frac{2^d}{d!}N_n^{d-1}(Q_n^*)^d+o_{\PP}(1).
\]
Since \(\log M_n=\log N_n-\log2+o(1)\),
Corollary~\ref{cor:gumbel-pair} gives the hitting-time claim.

If \eqref{eq:q-window} holds, then
\(N_n^{d-1}q_n^{d+1}=o(1)\), and \eqref{eq:r-q} gives
\[
    \frac{N_n^{d-1}}{d!}r(q_n)^d
    =
    \frac{2^d}{d!}N_n^{d-1}q_n^d+o(1)
    =
    \log M_n+c+\log2+o(1).
\]
Corollary~\ref{cor:exact-window} applied at \(c+\log2\) gives the
remaining assertions.
\end{proof}

\section{Exact representation geometry for diameter two}
\label{sec:d2}

For \(x\ne0\), let \(H_x\) be the graph with vertex set \(\cQ_G\),
where distinct classes \([a]\) and \([b]\) are adjacent if
\(x=\varepsilon a+\delta b\) for some
\(\varepsilon,\delta\in\{\pm1\}\).  Clearly \(H_x=H_{-x}\).

\begin{theorem}[Exact diameter-two representation geometry]
\label{thm:d2-geometry}
Let \(G\) be a finite abelian group of odd order \(N\), put
\(M=(N-1)/2\), and let \(x\ne0\) have order \(m\).  Then
\[
    H_x
    \cong
    P_{(m-1)/2}
    \sqcup
    \frac{N/m-1}{2}\,C_m,
\]
where \(P_s\) is the path on \(s\) vertices.  In particular,
\[
    |E(H_x)|=M-1.
\]
Every edge of the complete graph on \(\cQ_G\) belongs to precisely two
target graphs.  Moreover, if \([x]\ne[y]\), then \(H_x\) and \(H_y\)
have the unique common edge
\[
    \left\{
        \left[\frac{x+y}{2}\right],
        \left[\frac{x-y}{2}\right]
    \right\}.
\]
\end{theorem}

\begin{proof}
Since \(m\mid N\), the integer \(m\) is odd.  Put \(H=\langle x\rangle\).
The nonzero inversion classes contained in \(H\) are
\[
    [x],[2x],\ldots,\left[\frac{m-1}{2}x\right].
\]
Consecutive classes are adjacent because \((j+1)x-jx=x\).
Conversely, if \(1\le j<k\le(m-1)/2\) and \([jx]\) is adjacent to
\([kx]\), then
\[
    \varepsilon j+\delta k\equiv1\pmod m
\]
for some \(\varepsilon,\delta\in\{\pm1\}\).
Here \(1\le k-j\le(m-3)/2\) and \(3\le j+k\le m-2\); checking the
four signed combinations shows that necessarily \(k-j=1\).
Thus the classes in \(H\) induce \(P_{(m-1)/2}\).

Now let \(C=a+H\ne H\).  We have \(C\ne-C\): otherwise
\(2(a+H)=H\) in \(G/H\), and multiplication by \(2\), an automorphism
of the odd-order group \(G/H\), would give \(a+H=H\).
The pair \(C,-C\) yields the \(m\) distinct inversion classes
\[
    [a],[a+x],\ldots,[a+(m-1)x].
\]
Indeed, equality of two such classes gives either
\(j\equiv k\pmod m\) or \(2a+(j+k)x=0\); the latter would imply
\(a\in H\) by invertibility of doubling on \(G/H\).

Consecutive classes form a cycle.  If two distinct classes
\([a+jx]\) and \([a+kx]\) are adjacent, a same-sign representation
would again imply \(2a\in H\); hence the signs are opposite and
\((j-k)x=\pm x\).  Thus \(j-k\equiv\pm1\pmod m\), so there are no
other edges.  The \(N/m-1\) nontrivial cosets of \(H\) occur in pairs
\(\{C,-C\}\), proving the stated decomposition.  Its number of edges is
\[
    \frac{m-3}{2}+\frac{N/m-1}{2}m
    =
    \frac{N-3}{2}
    =
    M-1.
\]

For the overlap statement, let \([a]\ne[b]\).  The nonzero signed
sums of representatives determine exactly the target classes
\([a+b]\) and \([a-b]\).  They are distinct, since equality would
force \(2a=0\) or \(2b=0\).  Thus every edge belongs to exactly two
target graphs.

Finally, if \([x]\ne[y]\), set
\[
    a=\frac{x+y}{2},
    \qquad
    b=\frac{x-y}{2}.
\]
Then \(a,b\ne0\), \([a]\ne[b]\), and \(a+b=x\), \(a-b=y\).
Hence \(\{[a],[b]\}\) is common to \(H_x\) and \(H_y\).  Since every
edge belongs to only two target classes, it is the unique common edge.
\end{proof}

\begin{remark}
For \(d=2\), Theorem~\ref{thm:d2-geometry} sharpens the general
top-layer estimates: Proposition~\ref{prop:top-supports} gives only
\(e_2(x)=N/2+O(1)\), whereas here
\[
    e_2(x)=M-1=\frac{N-3}{2},
\]
and Lemma~\ref{lem:common-top} is sharpened from an \(O(1)\) overlap
to exactly one common support.
\end{remark}

\paragraph{Further remarks.}

Theorem~\ref{thm:main-ppp} identifies fixed diameter and odd order as
a rigid universality regime.  For even-order groups, inversion has
singleton as well as two-element orbits and multiplication by \(2\)
need not be invertible; the diameter-two threshold theory already
shows substantial dependence on involution structure
\cite{ChristofidesMarkstrom2014,ChristofidesMarkstromRange2014}.
For \(d=d_N\to\infty\), the constants hidden in the support and
dependency estimates must instead be quantified.  The threshold-scale
results of Christofides, Markstr\"om and Savvidou cover a substantial
growing-diameter range \cite{ChristofidesMarkstromSavvidou2026}, but
the centered extremal process in such a regime requires further
control of higher-order overlaps.

\end{document}